\documentclass[11pt,lenq]{amsart}
\usepackage{amsmath}
\usepackage{amscd}
\usepackage{amssymb}
\usepackage{amsbsy}
\usepackage{amsfonts}
\usepackage{latexsym}
\usepackage{graphics}
\usepackage{graphicx}
\usepackage{amsmath,amscd,latexsym}
\usepackage{multirow}
\usepackage{array}
\usepackage{paralist}
\usepackage{titletoc}

\theoremstyle{plain}
\newtheorem{theorem}{Theorem}[section]
\newtheorem{proposition}[theorem]{Proposition}
\newtheorem{lemma}[theorem]{Lemma}
\newtheorem{corollary}[theorem]{Corollary}
\newtheorem{remark}[theorem]{Remark}
\newtheorem{definition}[theorem]{Definition}
\newtheorem{notation}[theorem]{Notation}

\newtheorem{main theorem}[theorem]{Main Theorem}

\newtheorem{convention}[theorem]{Convention}

\newcommand{\NN}{\mathbb{N}}
\newcommand{\ZZ}{\mathbb{Z}}
\newcommand{\QQ}{\mathbb{Q}}
\newcommand{\RR}{\mathbb{R}}

\newcommand{\QQQ}{\hat{\mathbb{Q}}}

\newcommand{\Conway}{\mbox{\boldmath$S$}^{2}}
\newcommand{\Conways}
{(\mbox{\boldmath$S$}^{2},\mbox{\boldmath$P$})}
\newcommand{\PP}{\mbox{\boldmath$P$}}
\newcommand{\PConway}{\mbox{\boldmath$S$}}

\newcommand{\OO}{\mbox{\boldmath$O$}}

\newcommand{\rtangle}[1]{(B^3,t({#1}))}

\newcommand{\Isom}{\mbox{$\mathrm{Isom}$}}
\newcommand{\Fix}{\mbox{$\mathrm{Fix}$}}

\newcommand{\svert}{\,|\,}

\newcommand{\lp}{(\hskip -0.07cm (}
\newcommand{\rp}{)\hskip -0.07cm )}

\newcommand{\xr}{\mbox{$x_{\ell}$}}
\newcommand{\yr}{\mbox{$y_{\ell}$}}

\makeatletter
\renewcommand\subsection{\@startsection{subsection}{2}{0mm}
    {-10.5dd plus-8pt minus-4pt}{10.5dd}
     {\normalsize\upshape}}
\makeatother

\begin{document}
\title[Parabolic generating pairs of genus-one 2-bridge knots]
{Parabolic generating pairs of genus-one 2-bridge
knot groups}
\author{Donghi Lee}
\address{Department of Mathematics\\
Pusan National University \\
San-30 Jangjeon-Dong, Geumjung-Gu, Pusan, 609-735, Korea}
\email{donghi@pusan.ac.kr}

\author{Makoto Sakuma}
\address{Department of Mathematics\\
Graduate School of Science\\
Hiroshima University\\
Higashi-Hiroshima, 739-8526, Japan}
\email{sakuma@math.sci.hiroshima-u.ac.jp}

\subjclass[2010]{Primary 20F06, 57M25 \\
\indent {The first author was supported by Basic Science Research Program
through the National Research Foundation of Korea(NRF) funded
by the Ministry of Education, Science and Technology(2014R1A1A2054890).
The second author was supported by JSPS Grants-in-Aid 15H03620.}}


\begin{abstract}
We show that any parabolic generating pair of a
genus-one hyperbolic
$2$-bridge knot group
is equivalent to the upper or lower meridian pair.
As an application, we obtain a complete classification of
the epimorphisms from $2$-bridge knot groups to
genus-one hyperbolic
$2$-bridge knot groups.
\end{abstract}

\maketitle

\section{Introduction}
\label{sec:intro}

In \cite[Theorem~4.3]{Adams},
Adams proved that the fundamental group of a finite volume hyperbolic manifold is
generated by two parabolic transformations
if and only if it is homeomorphic to the complement of a $2$-bridge link $K(r)$
which is not a torus link.
This refines
the result of Boileau and Zimmermann \cite[Corollary~3.3]{Boileau-Zimmermann}
that a link in $S^3$ is a $2$-bridge link if and only if its link group is generated by two meridians.
Adams also proved that any parabolic generating pair
of a hyperbolic $2$-bridge link
consists of meridians.
This means that any such pair is represented by an arc
properly embedded in the exterior $E(K(r))$,
together with a pair of meridional loops on $\partial E(K(r))$
passing through the endpoints of the arc.
The meridian pair represented by the upper (resp., lower) tunnel
forms a parabolic generating pair, and is called
the {\it upper meridian pair} (resp., the {\it lower meridian pair}).
He also proved that
each hyperbolic $2$-bridge link group admits
only finitely many distinct parabolic generating pairs up to conjugacy
\cite[Corollary~4.1]{Adams} and moreover that,
for the figure-eight knot group,
the upper and lower meridian pairs are the only parabolic generating pairs
up to conjugacy \cite[Corollary~4.6]{Adams}).

These results were generalized to all $2$-bridge links by Agol \cite{Agol}.
In fact,
he classified all two parabolic generator Kleinian groups and their parabolic generating pairs.
To this end, he proved that
for any properly embedded arc in $E(K(r))$
which is not properly homotopic to the upper tunnel nor the lower tunnel,
the subgroup of the link group of $K(r)$
generated by the meridian pair represented by the arc
is a free group,
by using the checkerboard surfaces and Klein-Maskit combination theorem.

The purpose of this paper, however, is to give an alternative proof
of this result for
genus-one hyperbolic $2$-bridge knots
by using small cancellation theory and
a geometric observation suggested by Michel Boileau to us \cite{Boileau}.
Recall that a $2$-bridge knot has genus one
if and only if it is equivalent to $K(r)$ with
\[
r=[2m, \pm 2n]:=
\displaystyle \frac 1{\displaystyle 2m\pm \frac 1{\displaystyle 2n}},
\] where $m$ and $n$ are positive integers.

We now describe our strategy.
It is well-known that any parabolic generating pair
of a $2$-bridge link group $G(K(r))$ determines
a {\it strong inversion}, $h$, of the link $K(r)$,
i.e., an orientation-preserving involution of
$S^3$ preserving $K(r)$ setwise
such that the fixed point set $\Fix(h)$ is a circle
intersecting each component of $K(r)$ in two points.
The key observation, which Boileau brought to us,
is that
the parabolic generating pair is represented by
an arc component of $\Fix(h)\cap E(K(r))$
(see Proposition~\ref{prop:strong-involution_1}).

Every $2$-bridge link admits a diagram which has a $(\ZZ/2\ZZ)^2$-symmetry
as in Figure~\ref{fig.long_meridian_pair}.
Let $h_1$ (resp., $h_2$) be the $\pi$-rotations about the horizontal (resp., vertical) axis
in the projection plane.
If the slope $r=q/p$ satisfies the condition $q^2\not\equiv 1\pmod p$,
then
any strong inversion $h$ of $K(r)$
is strongly equivalent to
one of the two standard inversions $h_1$ and $h_2$,
namely $h$ is conjugate to $h_1$ or $h_2$ by
a homeomorphism of $(S^3,K(r))$ which is pairwise isotopic to the identity
(cf. \cite[Proposition~3.6]{Sakuma1} and the proof of Corollary~\ref{prop:strong-involution_2}).
We may assume that $\Fix(h_1)$ contains the upper tunnel, $\tau_1$, and
$\Fix(h_2)$ contains the lower tunnel, $\tau_2$.
(See \cite{Adams-Reid, Bleiler-Moriah, Futer} for interesting related results.)
Now suppose further that $K(r)$ is a knot,
and let $\tau_i'$ be the component
of $\Fix(h_i)\cap E(K(r))$ different from $\tau_i$.
We call the meridian pairs represented by
$\tau_1'$ and $\tau_2'$, respectively,
the {\it long upper meridian pair}
and the {\it long lower meridian pair}
(see Figure~\ref{fig.long_meridian_pair}).
The main ingredient of this paper is a combinatorial proof of the following theorem
based on small cancellation theory.

\begin{figure}[h]
\includegraphics{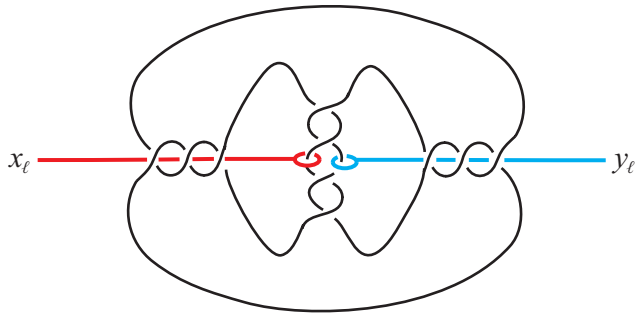}
\caption{
\label{fig.long_meridian_pair}
The long upper meridian pair
$\{\xr, \yr\}$
of $K(r)$ with $r=[4,6]$.
To be precise, $\xr$ (resp., $\yr$) is represented by the left (resp., right) lasso together with
an almost vertical line joining the end point of the lasso with the base point of $E(K(r))$.
Note that the upper tunnel is the short subarc, with both endpoints in $K(r)$,
of the horizontal central line in the projection plane.}
\end{figure}

\begin{theorem}
\label{thm:main_theorem}
Let $r=[2m,\pm 2n]$, where $m$ and $n$ are positive integers,
and let $(\xr, \yr)$ be the long upper meridian pair
or the long lower meridian pair for $K(r)$.
Then the subgroup of $G(K(r))$ generated by $\xr$ and $\yr$ is a free group.
\end{theorem}

Since there is a homeomorphism from $(S^3, K(r))$ with $r=[2m,\pm 2n]$ to
$(S^3,K(r'))$ with $r'=[2n,\pm 2m]$,
which maps the long lower meridian pair of $K(r)$
to the long upper meridian pair of $K(r')$,
it is enough to prove
the theorem only for the long upper meridian pair.
Thus throughout the remainder of this paper,
$(\xr, \yr)$ denotes the long upper meridian pair for $K(r)$
as illustrated in Figure~\ref{fig.long_meridian_pair}.

In the special case when $r=[2m,-2m]$,
we have yet another equivalence class of strong inversions,
which is represented by the strong inversion, $h_3$, illustrated by
Figure~\ref{fig.additional-inversion}
(see also Figure~\ref{fig.additional-inversion2} and \cite[Proposition~3.6]{Sakuma1}).

\begin{figure}[h]
\includegraphics{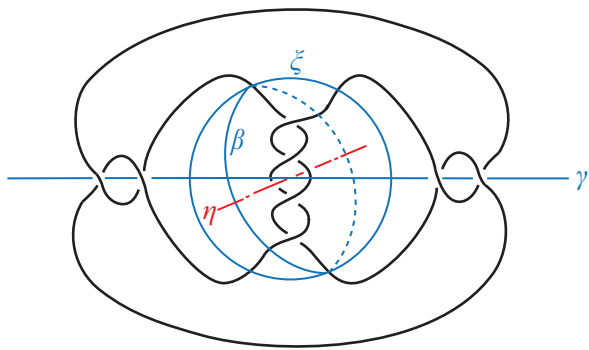}
\caption{
\label{fig.additional-inversion}
Additional symmetry of $K(r)$ for $r=[2m,-2m]$ ($m=2$).
$\Isom^+(S^3-K(r)) \cong \langle g, h_1 \svert g^4, h_1^2, (gh_1)^2\rangle$,
where $g=(\pi/2$-rotation about $\eta) \circ (\pi$-rotation about $\xi)$,
$h_1=\pi$-rotation around $\gamma$, and
$h_3=gh_1=\pi$-rotation around $\beta$.}
\end{figure}

\begin{theorem}
\label{thm:main_theorem_2}
Let $r=[2m,-2m]$, where $m$ is
an integer $\ge 2$,
and let $h_3$ be the strong inversion of $(S^3,K(r))$ as in the above.
Then, for each of the arc components of $\Fix(h_3)\cap E(K(r))$,
the subgroup of $G(K(r))$ generated by the meridian pair
represented by the arc is
a proper subgroup of $G(K(r))$.
\end{theorem}

In fact, it is not difficult to extend
Theorem~\ref{thm:main_theorem_2}
to all hyperbolic $2$-bridge knots $K(q/p)$ with $q^2\equiv 1 \pmod{p}$.
We also believe that Theorem~\ref{thm:main_theorem}
can be extended to all hyperbolic $2$-bridge links by a similar method,
but the arguments would become much more complicated.

These two theorems enable us to recover
a special case of Agol's result~\cite{Agol}.

\begin{theorem}
\label{cor:Agol-special}
Let $K(r)$ be a genus-one hyperbolic $2$-bridge knot,
namely $r=[2m,2n]$ with $m,n\in \NN$
or $r=[2m,-2n]$ with $m,n\in \NN$ and $(m,n)\ne (1,1)$.
Then the upper and lower meridian pairs are the only
parabolic generating pairs of the knot group of $K(r)$
up to equivalence.
\end{theorem}

For the precise definition of the equivalence relation in the above corollary,
see the first paragraph of Section~\ref{sec:strong-inversion}.

Together with the result of \cite[Corollary~1.3]{BBRW}
on epimorphisms from $2$-bridge knot groups
and
the characterization by \cite[Main Theorem~2.4]{lee_sakuma}
of upper meridian pair preserving epimorphisms
between $2$-bridge link groups,
the above corollary implies the following theorem.

\begin{theorem}
\label{thm:cor-epimorohism}
Let $K(r)$ be a genus-one hyperbolic $2$-bridge knot
and $K(\tilde r)$ be a $2$-bridge knot.
Then there is an epimorphism from $G(K(\tilde r))$ onto $G(K(r))$
if and only if
$\tilde r$ or $\tilde r+1$
belongs to
{\rm (i)} the $\hat\Gamma_r$-orbit of $r$ or $\infty$
or
{\rm (ii)} the $\hat\Gamma_{r'}$-orbit of $r'$ or $\infty$.
Here {\rm (a)} $\hat\Gamma_r$ {\rm(}resp., $\hat\Gamma_{r'}${\rm)}
is the subgroup of the automorphism group
of the Farey tessellation
generated by the reflections in the Farey edges with endpoints $\infty$ or $r$
{\rm(}resp., $r'${\rm)},
and {\rm (b)} $r'= q'/ p$,
where $q q'\equiv 1 \pmod p${\rm ;}
$p$ and $q$ are relatively prime integers such that
$r=q/p$.
\end{theorem}

At the end of the introduction, we would like to point out that
if we remove
the condition that generating pairs consist of parabolic elements,
then it is proved by Heusener and Porti~\cite{Heusener-Porti} that
every hyperbolic knot admits infinitely many generating pairs
up to Nielsen equivalence.
Moreover, the same conclusion for torus knots,
especially non-hyperbolic $2$-bridge knots,
had been proved by Zieschang~\cite{Zieschang}.

The authors would like thank Michel Boileau for
fruitful discussions.
They would also like to thank Colin Adams and Ian Agol for their encouragement.

\section{Strong inversions associated with parabolic generating pairs}
\label{sec:strong-inversion}

Let $K$ be a link in $S^3$ and $E(K)$ the exterior of $K$,
namely the complement of an open regular neighborhood of $K$.
An essential simple loop in $\partial E(K)$ is called a {\it meridian}
if it bounds a disk on the (closed) regular neighborhood of $K$.
An element of the link group $G(K)=\pi_1(E(K))$ which is freely homotopic
to a meridian is also called a {\it meridian}.
By a {\it meridian pair} of $K$,
we mean an unordered pair $\{x,y\}$ of meridians of $G(K)$.
Two meridian pairs $\{x,y\}$ and $\{x',y'\}$ are said to be {\it equivalent}
if $\{x',y'\}$ is equal to $\{x^{\varepsilon_1}, y^{\varepsilon_2}\}$
for some $\varepsilon_1, \varepsilon_2\in\{\pm 1\}$
up to simultaneous conjugacy.

Note that there is a bijective correspondence between the
set of meridian pairs up to equivalence
and the set of arcs properly embedded in $E(K)$ up to proper homotopy.
Here a proper arc $\gamma$ in $E(K)$ corresponds to
a meridian par $\{x, y\}$ which is obtained as follows.
Pick an interior point, $q$, in $\gamma$,
and divide $\gamma$ into two subarcs $\gamma_1$ and $\gamma_2$
such that $\gamma_1\cap \gamma_2=\{q\}$.
For $i=1,2$, let $C_i$ be a meridian passing through the end point
of $\gamma_i$ in $\partial E(K)$.
Let $\delta$ be an arc in $E(K)$ joining $q$ with
the base point of $E(K)$.
Then the pair of based loops $\{\delta\cup \gamma_i\cup C_i\}_{i=1,2}$
gives the meridian pair $\{x,y\}$ corresponding to the arc $\gamma$.

If $K$ is hyperbolic, then a meridian pair $\{x,y\}$ is identified with
a pair of parabolic transformations of the hyperbolic $3$-space.
We are interested only in the case where $\{x,y\}$ generate a non-elementary
Kleinian group, namely the case where
the parabolic fixed points of $x$ and $y$ are not identical.
Then the geodesic joining the parabolic fixed points of $x$ and $y$
descends to a proper geodesic in the hyperbolic manifold $S^3-K$
and hence determines a proper arc in $E(K)$,
where we identify $E(K)$ with the complement of
an open cusp neighborhoods.
By the correspondence between the fundamental group and
the covering transformation group,
we see that this arc corresponds to the meridian pair $\{x, y\}$.

Boileau~\cite{Boileau} informed us of the following fact,
which had been observed by Adams~\cite{Adams}.

\begin{proposition}
\label{prop:strong-involution_1}
Let $K(r)$ be a hyperbolic $2$-bridge link,
and let $\{x,y\}$ be a
parabolic generating pair of the link group $G(K(r))$.
Then there is a strong inversion $h$ of $K(r)$
such that $(h_*(x), h_*(y))=(x^{-1}, y^{-1})$ and that
$\{x,y\}$ is a meridian pair
corresponding to an arc component
of $\Fix(h)\cap E(K(r))$.
Here $h_*$ denotes the automorphism of $G(K(r))$ induced by $h$.
\end{proposition}

\begin{proof}
Let $K(r)$ be a hyperbolic $2$-bridge link,
and let $\{x,y\}$ be a parabolic generating pair.
Then, by assumption, $x$ and $y$ are identified with
parabolic transformations.
Since $\{x,y\}$ generates non-elementary group $G(K(r))$,
$x$ and $y$ have distinct parabolic fixed points.
Let $\eta$ be the order $2$ elliptic transformation
whose axis is the geodesic, $\tilde c$, joining the two parabolic fixed points.
Then $\eta x\eta=x^{-1}$ and
$\eta y\eta=y^{-1}$ (cf. \cite[Section~5.4]{Thurston})
and therefore $\eta$ descends to an
orientation-preserving involution, $h$, of $S^3-K(r)$,
such that the restriction of $h$ to $\partial E(K(r))$
is a hyper-elliptic involution.
Thus $h$ extends to a strong inversion of $K(r)$,
which we continue to denote by the same symbol $h$.
Now recall the result \cite[Theorem~4.3]{Adams} that
$\{x,y\}$ is a meridian pair.
Thus by the observation made before this proposition,
we obtain the desired result.
\end{proof}

\begin{corollary}
\label{prop:strong-involution_2}
For a hyperbolic $2$-bridge knot $K(r)$ with $r=q/p$,
the following hold.

{\rm (1)}
If $q^2\not\equiv 1\pmod p$,
then any parabolic generating pair of $G(K(r))$
is equivalent to either the upper, lower,
long upper or long lower meridian pair.

{\rm (2)}
If $q^2\equiv 1\pmod p$,
then for any parabolic generating pair of $G(K(r))$,
one of the following holds.
\begin{enumerate}
\item[{\rm (i)}]
It is equivalent to either the upper, lower,
long upper or long lower meridian pair.
\item[{\rm (ii)}]
There is an automorphism of $G(K(r))$
which carries it to a parabolic generating pair
represented by one of the two arc components
of $\Fix(h_3)\cap E(K(r))$,
where $h_3$ is the strong inversion, as illustrated in
Figure~\ref{fig.additional-inversion}.
\end{enumerate}
\end{corollary}

\begin{proof}
Suppose that
$q^2\not\equiv 1\pmod p$.
Then the orientation-preserving isometry group
$\Isom^+(S^3-K(r))$ is isomorphic to
$\ZZ/2\ZZ\oplus\ZZ/2\ZZ$ (see \cite{Bonahon-Siebenmann, Sakuma2}).
By Tollefson's theorem~\cite{Tollefson} or
the orbifold theorem~\cite{Boileau-Porti, CHK},
this implies that $h$ is strongly equivalent to
one of the two strong inversions
$h_1$ and $h_2$ introduced in the introduction.
Hence, we obtain the conclusion by Proposition~\ref{prop:strong-involution_1}.

Suppose that
$q^2\equiv 1\pmod p$.
Then $\Isom^+(S^3-K(r))$ is isomorphic to
the dihedral group $D_8$ of order $8$ (see \cite{Bonahon-Siebenmann, Sakuma2}).
To be precise,
$\Isom^+(S^3-K(r))\cong \langle g, h_1 \svert g^4, h_1^2, (gh_1)^2\rangle$,
where $g$ and $h_1$ are as illustrated in Figure~\ref{fig.additional-inversion}.
By using this fact and Tollefson's theorem \cite{Tollefson} or
the orbifold theorem~\cite{Boileau-Porti, CHK},
we see that any strong inversion $h$ of $K(r)$ is strongly equivalent to $g^ih_1$ with $0\le i\le 3$.
Since $g^0h_1=h_1$, $gh_1=h_3$, $g^2h_1=gh_1g^{-1}(=h_2)$ and $g^3h_1=gh_3g^{-1}$,
we obtain the conclusion by Proposition~\ref{prop:strong-involution_1}.
\end{proof}

\section{Wirtinger generators and long upper/lower meridian pairs}
\label{sec:Wirtinger-generator}

Throughout this paper, we use the following notation:
For an element $x$ in a group, we denote $x^{-1}$ by $\bar x$.
For two elements $x$ and $y$ of a group and for a positive integer $k$,
we define $\langle xy\rangle^k$ to be the alternative product of $x$ and $y$ of length $k$.
Namely:
\[
\langle xy\rangle^k
=
\begin{cases}
(xy)^{\frac{k}{2}}
& \text{if $k$ is even};\\
(xy)^{\frac{k-1}{2}}x
& \text{if $k$ is odd}.
\end{cases}
\]
We also define $\langle xy\rangle^{-k}$ to be
$(\langle xy\rangle^k)^{-1}$.

Consider the genus-one $2$-bridge knot $K(r)$ with $r=[2m, \pm 2n]$,
where $m$ and $n$ are positive integers.
Let $c_i \ (-m \le i \le m+1)$ and $d_j \ (0 \le j \le 2n)$
be the Wirtinger generators of the knot group $G(K(r))$
as illustrated in Figure~\ref{fig.knot_diagram}(a) and (b)
according to whether $r=[2m,2n]$ and $[2m,-2n]$, respectively.
Here,
we follow the convention of \cite{Crowell-Fox}.
Namely, we assume that the base point of $E(K(r))$ lies far above the
projection plane, and
the symbol, say $c_i$, denotes the element of the
knot group represented by an oriented short arc passing under
the arc of the knot diagram with label $c_i$
in a {\it left-right} direction,
together with a pair of straight almost vertical arcs
joining the endpoints of the short arc with the base point of $E(K(r))$.

Set
\begin{align}
\label{generators-ab}
a:= c_{1} \quad \textrm{and} \quad b:=c_{0}^{-1}
\end{align}
Then $\{a,b\}$ is a generating pair of $G(K(r))$
which is identical with that
of $G(K(r))$ in {\rm \cite[Section~3]{lee_sakuma}}.
To see this, let $B_1$ be a small regular neighborhood of
the upper tunnel,
and $B_2$ be the closure of the complement of $B_1$ in $S^3$.
Then $(B_1,B_1\cap K(r))$ and $(B_2, B_2\cap K(r))$
are identified with the rational tangles of slopes $\infty$ and $r$, respectively.
Thus, from the description of the generator system in \cite[Section~3]{lee_sakuma},
we see that $\{a,b\}$ is the generating pair given in it.
Hence
\[
G(K(r))\cong \langle a,b \svert u_r \rangle,
\]
where $u_r$ is the cyclically reduced word in $a$ and $b$ in {\rm \cite[Lemma~3.1]{lee_sakuma}}.
Namely,
\begin{align}
\label{group-presentation}
u_{r}=a\hat{u}_{r}b\hat{u}_{r}^{-1}
\quad
\text{with}
\quad
\hat{u}_{r} = b^{\varepsilon_1} a^{\varepsilon_2} \cdots b^{\varepsilon_{p-2}} a^{\varepsilon_{p-1}},
\end{align}
where $r=q/p$ with $(p,q)=(4mn\pm 1,2n)$ and $\varepsilon_i = (-1)^{\lfloor iq/p \rfloor}$.
Here $\lfloor t \rfloor$ is the greatest integer not exceeding $t$.

\begin{figure}[h]
\includegraphics{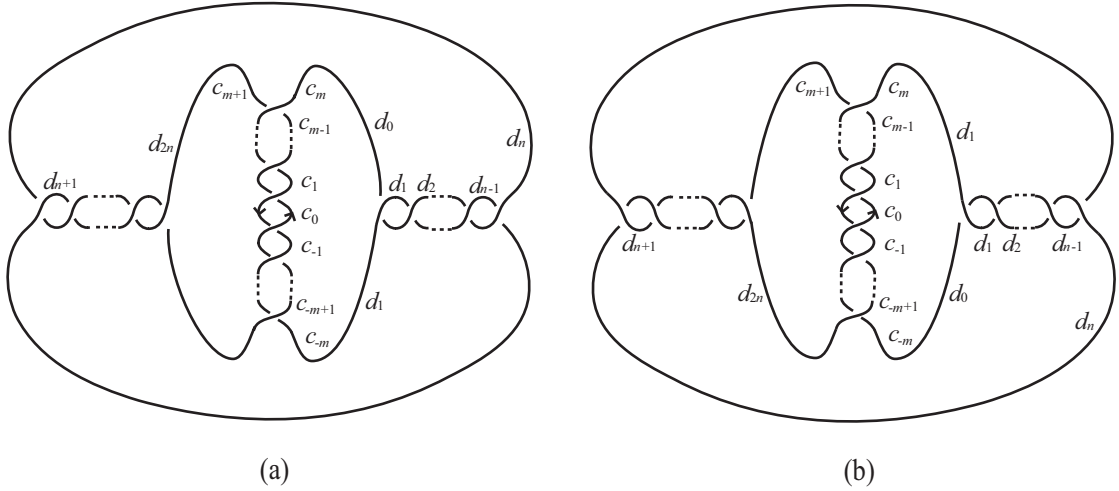}
\caption{
\label{fig.knot_diagram}
Wirtinger generators $c_i \ (-m \le i \le m+1)$ and $d_j \ (0 \le j \le 2n)$
of the knot group $G(K(r))$ with (a) $r=[2m,2n]$ and (b) $r=[2m,-2n]$.
The orientations of the overpasses are determined by those of the two
central overpasses. In particular, the orientations of the overpasses with label
$d_0$ and $d_n$ depend on the parities of $m$ and $n$.
}
\end{figure}

Let $f:=h_1h_2$ be the involution of $(S^3,K(r))$ induced by
the $\pi$-rotation about the axis which intersects the projection plane,
in Figure~\ref{fig.knot_diagram}, orthogonally at the central point.
We also use the same symbol $f$ to denote the automorphism of $G(K(r))$ induced by
the involution $f$.
Recall that $(\xr,\yr)$ denotes the long upper meridian pair of $G(K(r))$
as illustrated in Figure~\ref{fig.long_meridian_pair}.
Then we can easily observe the following lemma.

\begin{lemma}
\label{lem:group-symmetry}
We have $(f(a),f(b))=(b^{-1},a^{-1})$, $(f(\xr),f(\yr))=(\yr,\xr)$ and
\[
f(c_i)=c_{1-i} \quad (-m\le i \le m+1), \qquad
f(d_j)=d_{2n+1-j} \quad (1\le j\le 2n).
\]
\end{lemma}

The following lemma can be easily verified
by a standard calculation on Wirtinger presentation,
where the second formula is obtained from the first formula
by using Lemma~\ref{lem:group-symmetry}.

\begin{lemma}
\label{lem:Wirtnger-generator-c}
For $1\le i\le m+1$, the following hold.
\[
c_i=
\begin{cases}
\langle \bar a \bar b\rangle^{i}
\langle a b\rangle^{i-1}
& \text{if $i$ is even};\\
\langle \bar a \bar b\rangle^{i-1}
\langle a b\rangle^{i}
& \text{if $i$ is odd}.
\end{cases}
\]
For $1\le i\le m$, the following hold.
\[
c_{-i}=
\begin{cases}
\langle b a\rangle^{i}
\langle \bar b \bar a\rangle^{i+1}
& \text{if $i$ is even};\\
\langle b a \rangle^{i+1}
\langle \bar b \bar a\rangle^{i}
& \text{if $i$ is odd}.
\end{cases}
\]
\end{lemma}

As shown in Figure~\ref{fig.knot_diagram}, we have
\begin{align}
\label{d0d1}
(d_0,d_1)
=
\begin{cases}
(c_m,c_{-m})
& \text{if $r=[2m,2n]$};\\
(c_{-m},c_m)
& \text{if $r=[2m,-2n]$}.
\end{cases}
\end{align}

\begin{lemma}
\label{lem:pre_longest-base}
The member $\yr$ of the long upper meridian pair $(\xr,\yr)$
is given by the following formula.
\[
\yr=
\begin{cases}
\langle d_1\bar d_0\rangle^{n} \bar b \langle d_1\bar d_0\rangle^{-n}
& \text{if $m$ is even},\\
\langle \bar d_1 d_0\rangle^{n} \bar b \langle \bar d_1 d_0\rangle^{-n}
& \text{if $m$ is odd}.
\end{cases}
\]
The other member $\xr$ is obtained from $\yr$ by replacing $(a,b)$
with $(b^{-1},a^{-1})$.
\end{lemma}

\begin{proof}
We prove the formula only in the case
where $r=[2m,\pm 2n]$ with $m$ even.
(The other case is proved similarly.)
Observe from Figures~\ref{fig.long_meridian_pair} and \ref{fig.knot_diagram}
that $\yr=w \bar b \bar w$, where
\[
w=
\begin{cases}
(\bar d_n d_{n-1}) \cdots (\bar d_2 d_1)
&\text{if $n \ge 2$ is even};\\
(d_n \bar d_{n-1}) \cdots (d_3 \bar d_2) d_1
& \text{if $n \ge 1$ is odd}.
\end{cases}
\]

On the other hand, we can observe from Figure~\ref{fig.knot_diagram}
that $\bar d_{2j} d_{2j-1}=d_{2j-1} \bar d_{2j-2}= d_1\bar d_0$.
Hence we obtain the formula.
The last assertion is a consequence of Lemma~\ref{lem:group-symmetry}.
\end{proof}

By using Lemmas~\ref{lem:Wirtnger-generator-c}, \ref{lem:pre_longest-base}
and the identities (\ref{generators-ab}), (\ref{d0d1}),
we can express $d_0$, $d_1$, $\xr$ and $\yr$ as words in $\{a,b\}$.

\begin{convention}
\label{con:words}
{\rm
In the remainder of this paper,
we use the symbols $d_0$, $d_1$, $\xr$ and $\yr$
to denote the {\it reduced} words in $\{a,b\}$
obtained as in the above.}
\end{convention}

For the definition of reduced words, see Section~\ref{sec:S-sequence} below.

\section{$S$-sequences of long upper/lower meridian pairs}
\label{sec:S-sequence}

We first recall basic terminology in combinatorial group theory and the definition
of $S$-sequences
and cyclic $S$-sequences introduced in \cite[Section~4]{lee_sakuma}.
Let $X$ be a set.
By a {\it word} in $X$, we mean a finite sequence
$x_1^{\varepsilon_1}x_2^{\varepsilon_2}\cdots x_t^{\varepsilon_t}$
where $x_i\in X$ and $\varepsilon_i=\pm 1$.
We call $x_1^{\varepsilon_1}$ and $x_t^{\varepsilon_t}$
the {\it initial letter} and {\it terminal letter} of the word, respectively.
For two words $u, v$ in $X$, by
$u \equiv v$ we denote the {\it visual equality} of $u$ and
$v$, meaning that if $u=x_1^{\varepsilon_1} \cdots x_t^{\varepsilon_t}$
and $v=y_1^{\delta_1} \cdots y_m^{\delta_m}$ ($x_i, y_j \in X$; $\varepsilon_i, \delta_j=\pm 1$),
then $t=m$ and $x_i=y_i$ and $\varepsilon_i=\delta_i$ for each $i=1, \dots, t$.
The length of a word $v$ is denoted by $|v|$.
A word $v$ in $X$ is said to be {\it reduced} if $v$ does not contain $xx^{-1}$ or $x^{-1}x$ for any $x \in X$.
A {\it cyclic word} is defined to be the set of all cyclic permutations of a
cyclically reduced word. By $(v)$ we denote the cyclic word associated with a
cyclically reduced word $v$.
Also by $(u) \equiv (v)$ we mean the {\it visual equality} of two cyclic words
$(u)$ and $(v)$. In fact, $(u) \equiv (v)$ if and only if $v$ is visually a cyclic shift
of $u$.

\begin{definition}
\label{def:alternating}
{\rm (1) Let $v$ be a reduced word in
$\{a,b\}$. Decompose $v$ into
\[
v \equiv v_1 v_2 \cdots v_t,
\]
where, for each $i=1, \dots, t-1$, all letters in $v_i$ have positive (resp., negative) exponents,
and all letters in $v_{i+1}$ have negative (resp., positive) exponents.
Then the sequence of positive integers
$S(v):=(|v_1|, |v_2|, \dots, |v_t|)$ is called the {\it $S$-sequence of $v$}.

(2) Let $(v)$ be a cyclic word in
$\{a, b\}$. Decompose $(v)$ into
\[
(v) \equiv (v_1 v_2 \cdots v_t),
\]
where all letters in $v_i$ have positive (resp., negative) exponents,
and all letters in $v_{i+1}$ have negative (resp., positive) exponents (taking
subindices modulo $t$). Then the {\it cyclic} sequence of positive integers
$CS(v):=\lp |v_1|, |v_2|, \dots, |v_t| \rp$ is called
the {\it cyclic $S$-sequence of $(v)$}.
Here the double parentheses denote that the sequence is considered modulo
cyclic permutations.

(3) A reduced word $v$ in $\{a,b\}$ is said to be {\it alternating}
if $a^{\pm 1}$ and $b^{\pm 1}$ appear in $v$ alternately,
i.e., neither $a^{\pm2}$ nor $b^{\pm2}$ appears in $v$.
A cyclic word $(v)$ is said to be {\it alternating}
if all cyclic permutations of $v$ are alternating.
In the latter case, we also say that $v$ is {\it cyclically alternating}.
}
\end{definition}

\begin{notation}
\label{notation:alternating}
\rm{
For $x,y \in \{a^{\pm 1}, b^{\pm 1}\}$
and for a word $w$ in $a,b$,
we write
$w\equiv x\cdots y$
if $w$ is an {\it alternating} word
which begins with $x$ and ends with $y$.
It should be noted that an alternating word
is determined by the initial letter and the $S$-sequence.
}
\end{notation}

In the remainder of this section,
we describe the $S$-sequences of
the words $d_0$, $d_1$, $\xr$ and $\yr$ in Convention~\ref{con:words}.

\begin{lemma}
\label{lem:longest-base}
Suppose $r=[2m,2n]$.
Then $d_0\equiv a$ or $\bar a\cdots a$
according to whether $m=1$ or $m\ge 2$,
and $d_1\equiv b \cdots\bar b$.
Moreover,
their $S$-sequences are given as follows.
\begin{enumerate}[\indent \rm (1)]
\item If $m$ is even, then
$S(d_0)=(m, m-1)$ and $S(d_1)=(m, m+1)$.
\item If $m$ is odd, then
$S(d_0)=(1)$ or $(m-1, m)$
according to whether $m=1$ or $m \ge 3$,
and $S(d_1)=(m+1,m)$.
\end{enumerate}
\end{lemma}

\begin{proof}
By the formula~(\ref{d0d1}) and Lemma~\ref{lem:Wirtnger-generator-c}, we have
the following identity in the free group $F(a,b)$ with free basis $\{a,b\}$:
\[
d_0=c_m=
\begin{cases}
\langle \bar a \bar b\rangle^{m}
\langle a b\rangle^{m-1}
& \text{if $m$ is even};\\
\langle \bar a \bar b\rangle^{m-1}
\langle a b\rangle^{m}
& \text{if $m$ is odd}.
\end{cases}
\]
Since the last two words on the right hand side in the above identity are alternating
and therefore reduced, we obtain the assertion for the reduced word $d_0$.
Similarly, by the formula~(\ref{d0d1}) and
Lemma~\ref{lem:Wirtnger-generator-c} again,
we have the following identity in $F(a,b)$:
\[
d_1=c_{-m}=
\begin{cases}
\langle b a\rangle^{m}
\langle \bar b \bar a\rangle^{m+1}
& \text{if $m$ is even};\\
\langle b a\rangle^{m+1}
\langle \bar b \bar a\rangle^{m}
& \text{if $m$ is odd}.
\end{cases}
\]
Hence
we obtain the assertion for $d_1$.
\end{proof}

\begin{lemma}
\label{lem:longest-base_PLUS}
Suppose $r=[2m,-2n]$.
Then $d_0\equiv b \cdots\bar b$ and
$d_1\equiv a$ or $\bar a\cdots a$ according to whether $m=1$ or $m \ge 2$.
Moreover, their $S$-sequences are given as follows.
\begin{enumerate}[\indent \rm (1)]
\item If $m$ is even, then
$S(d_0)=(m, m+1)$ and $S(d_1)=(m, m-1)$.
\item If $m$ is odd, then
$S(d_0)=(m+1,m)$
and $S(d_1)=(1)$ or $(m-1, m)$
according to whether $m=1$ or $m \ge 3$.
\end{enumerate}
\end{lemma}

\begin{proof}
This is proved by using
the formula~(\ref{d0d1}) and
Lemma~\ref{lem:Wirtnger-generator-c}
as in Lemma~\ref{lem:longest-base}.
\end{proof}

\begin{proposition}
\label{prop:longest-base}
Suppose $r=[2m,2n]$, and set $\varepsilon=(-1)^n$.
Then the following hold.

{\rm (1)}
$\xr\equiv w_xa \bar w_x$ and $\yr\equiv w_y\bar b \bar w_y$,
where $w_x$ and $w_y$ are the alternating words such that
$w_x\equiv \bar a \cdots \bar b^{\varepsilon}$, $w_y\equiv b \cdots a^{\varepsilon}$
and
\[
S(w_x)=S(w_y)=
\begin{cases}
(m, (n-1)\langle 2m\rangle, m) & \textrm{if $n \ge 2$};\\
(m, m) & \textrm{if $n=1$}.
\end{cases}
\]

{\rm (2)}
$\xr\equiv \bar{a} \cdots a$, $\yr\equiv b \cdots \bar{b}$, and
\[
S(\xr^{\varepsilon})=S(\yr^{\varepsilon})=
\begin{cases}
(m, (n-1)\langle 2m\rangle, m, m+1, (n-1)\langle 2m\rangle, m) & \textrm{if $n \ge 2$};\\
(m, m, m+1, m) & \textrm{if $n=1$}.
\end{cases}
\]

{\rm (3)}
For any non-zero integer $k$,
the words $w_x a^k \bar w_x$ and $w_y \bar b^k \bar w_y$
are reduced,
and we have the identities
$\xr^k=w_x a^k \bar w_x$ and $\yr^k=w_y \bar b^k \bar w_y$ in $F(a,b)$.
Moreover, each of $w_x a^k \bar w_x$ and $w_y \bar b^k \bar w_y$ is
alternating if and only if $|k|=1$.
\end{proposition}

In the above proposition,
the symbol ``$k \langle 2m\rangle$'' represents $k$ successive $2m$'s.

\begin{proof}
Note that $\xr=f(\yr)$
is obtained from $\yr$ by replacing $a^{\pm1}$ with $b^{\mp 1}$
by Lemma~\ref{lem:group-symmetry}.
Moreover, the assertion (3) is an immediate consequence of
the assertions (1) and (2).
Thus we prove (1) and (2) for $\yr$.

Suppose first that $m$ is even.
Then $\yr=\langle d_1\bar d_0\rangle^{n} \bar b \langle d_1\bar d_0\rangle^{-n}$
in $F(a,b)$
by Lemma~\ref{lem:pre_longest-base}.
By Lemma~\ref{lem:longest-base}, we see
$d_1\bar d_0\equiv (b\cdots \bar{b})(\bar{a} \cdots a)\equiv b \cdots a$
and $S(d_1 \bar{d_0})=(m,2m,m)$.

If $n$ is even,
then these identities imply that
$\langle d_1\bar d_0\rangle^{n}\equiv b \cdots a$ and
$S(\langle d_1\bar d_0\rangle^{n})=(m, (n-1)\langle 2m\rangle, m)$.
(This follows from the following fact:
Since the initial letter and the terminal letter of $d_1\bar d_0$
are $b$ and $a$, respectively,
the terminal component, $m$, of the $S$-sequence of the $i$-th factor $d_1\bar d_0$
is amalgamated with the initial component, $m$,
of the $S$-sequence of the $i+1$-th factor $d_1\bar d_0$ to form a component $2m$
of $S(\langle d_1\bar d_0\rangle^{n})$ for each $i$ with $1\le i\le n/2$.)
Set $w_y$ to be the alternating word $\langle d_1\bar d_0\rangle^{n}$
in $\{a,b\}$.
Then $\yr=w_y\bar b \bar w_y$ in $F(a,b)$
and $w_y\bar b \bar w_y\equiv (b \cdots a)\bar b (\bar a \cdots \bar b)\equiv b \cdots \bar b$
is alternating.
Hence $\yr\equiv w_y\bar b \bar w_y$.
Moreover, we can observe that $S(w_y)$ and $S(\yr)$ are of the desired form.

If $n$ is odd and $\ge 3$,
then
$\langle d_1\bar d_0\rangle^{n}\equiv
\langle d_1\bar d_0\rangle^{n-1}d_1
\equiv (b \cdots a)(b \cdots \bar a\bar b)\equiv b \cdots \bar a\bar b$.
(Here, we extend Notation~\ref{notation:alternating} so that
this identity means that $\langle d_1\bar d_0\rangle^{n}$ is an alternating
word in $\{a,b\}$ which begins with $b$ and end with $\bar a\bar b$.)
Moreover,
$S(\langle d_1\bar d_0\rangle^{n})=(m, (n-1)\langle 2m\rangle, m+1)$.
Let $w_y$ be the alternating word obtained from
the alternating word $\langle d_1\bar d_0\rangle^{n}$ by deleting
the last letter $\bar b$.
Then $w_y\equiv b \cdots \bar a$,
$S(w_y)=(m, (n-1)\langle 2m\rangle, m)$
and
$\yr=\langle d_1\bar d_0\rangle^{n} \bar b \langle d_1\bar d_0\rangle^{-n}
=(b \cdots \bar b) \bar b (b \cdots \bar b)$
reduces to the alternating word
$w_y \bar b \bar w_y
\equiv (b \cdots \bar a) \bar b (a \cdots \bar b)\equiv b \cdots \bar b$.
Hence $\yr\equiv w_y\bar b \bar w_y$.
Moreover, we can observe that $S(w_y)$ and $S(\yr)$ are of the desired form.
These arguments also work even when $n=1$, if we discard
the entry $(n-1)\langle 2m\rangle=0\langle 2m\rangle$.

Suppose next that $m$ is odd.
Then $\yr= \langle \bar d_1 d_0\rangle^{n} \bar b \langle \bar d_1 d_0\rangle^{-n}$
in $F(a,b)$ by Lemma~\ref{lem:pre_longest-base}.
By Lemma~\ref{lem:longest-base},
we see
$\bar d_1 d_0\equiv (b\cdots \bar{b})(\bar{a} \cdots a)=b \cdots a$
and $S(\bar d_1 d_0)=(m,2m,m)$.
By similar arguments,
we obtain the desired result
by setting $w_y$ to be the alternating word $ \langle \bar d_1 d_0\rangle^{n}$ or
the alternating word obtained from $ \langle \bar d_1 d_0\rangle^{n}$ by removing
the last letter
according to whether $n$ is even or odd.
\end{proof}

\begin{proposition}
\label{prop:longest-base_PLUS}
Suppose $r=[2m,-2n]$, and set $\varepsilon=(-1)^n$.
Then the following hold.

{\rm (1)}
$\xr\equiv w_xa \bar w_x$ and $\yr\equiv w_y\bar b \bar w_y$,
where $w_x$ and $w_y$ are the alternating words such that
$w_x\equiv b \cdots b^{\varepsilon}$,
$w_y\equiv \bar a \cdots \bar a^{\varepsilon}$,
except when $(m,n)=(1,1)$,
and such that
\[
S(w_x)=S(w_y)=
\begin{cases}
(m, (n-1)\langle 2m\rangle, m-1) & \textrm{if $m\ge 2$ and $n\ge2$}; \\
(m, m-1) & \textrm{if $m\ge 2$ and $n=1$};\\
(1, (n-1)\langle 2\rangle) & \textrm{if $m=1$ and $n \ge 2$};\\
(1) & \textrm{if $m=1$ and $n =1$}.
\end{cases}
\]
In the exceptional case when $(m,n)=(1,1)$,
we have $w_x\equiv b$ and $w_y\equiv \bar a$.

{\rm (2)}
$\xr\equiv b \cdots \bar b$, $\yr\equiv \bar a \cdots a$, and
\begin{enumerate}[\indent \rm (i)]
\item
If $m \ge 2$, then
\[
S(\xr^{\varepsilon})=S(\yr^{\varepsilon})=
\begin{cases}
(m, (n-1)\langle 2m\rangle, m, m-1, (n-1)\langle 2m\rangle, m) &
\textrm{if $n \ge 2$};\\
(m, m, m-1, m) & \textrm{if $n=1$}.
\end{cases}
\]
\item
If $m=1$, then
\[
S(\xr^{\varepsilon})=S(\yr^{\varepsilon})=
\begin{cases}
(1, (n-1)\langle 2\rangle, 3, (n-2)\langle 2\rangle, 1) & \textrm{if $n \ge 3$};\\
(1, 2, 3, 1) & \textrm{if $n=2$};\\
(1,2) & \textrm{if $n=1$}.
\end{cases}
\]
\end{enumerate}

{\rm (3)}
For any non-zero integer $k$,
the words $w_x a^k \bar w_x$ and $w_y \bar b^k \bar w_y$
are reduced,
and we have the identities
$\xr^k=w_x a^k \bar w_x$ and $\yr^k=w_y \bar b^k \bar w_y$ in $F(a,b)$.
Moreover, each of $w_x a^k \bar w_x$ and $w_y \bar b^k \bar w_y$ is
alternating if and only if $|k|=1$.
\end{proposition}

\begin{proof}
As in the proof of Proposition~\ref{prop:longest-base},
we have only to prove the assertions (1) and (2) for $\yr$.

Suppose first that $m$ is even.
Then $\yr=\langle d_1\bar d_0\rangle^{n} \bar b \langle d_1\bar d_0\rangle^{-n}$
in $F(a,b)$ by Lemma~\ref{lem:pre_longest-base}.
By Lemma~\ref{lem:longest-base}, we see
$d_1\bar d_0\equiv (\bar a\cdots a)(b \cdots \bar b)\equiv \bar a \cdots \bar b$
and $S(d_1 \bar{d_0})=(m,2m,m)$.

If $n$ is even,
then these identities imply that
$\langle d_1 \bar{d_0}\rangle^{n}\equiv \bar a \cdots \bar a\bar b$ and
$S(\langle d_1 \bar{d_0}\rangle^{n})=(m, (n-1)\langle 2m\rangle, m)$.
Set $w_y$ to be the alternating word in $\{a,b\}$
obtained from the alternating word $\langle d_1\bar d_0\rangle^{n}$
by deleting the last letter $\bar b$.
Then $w_y\equiv \bar a \cdots \bar a$,
$S(w_y)=(m, (n-1)\langle 2m\rangle, m-1)$
and
$\yr=\langle d_1 \bar{d_0}\rangle^{n} \bar b \langle d_1 \bar{d_0}\rangle^{-n}
=(a \cdots \bar b) \bar b (b \cdots \bar a)$
reduces to the alternating word
$w_y \bar b \bar w_y
\equiv (a \cdots \bar a) \bar b (a \cdots \bar a)\equiv a \cdots \bar a$.
Hence $\yr\equiv w_y \bar b \bar w_y$.
Moreover, we can observe that $S(w_y)$ and $S(\yr)$ are of the desired form.

If $n$ is odd and $\ge 3$, then we have
$\langle d_1 \bar{d_0}\rangle^{n}\equiv
\langle d_1 \bar{d_0}\rangle^{n-1} d_1
\equiv (\bar a \cdots \bar b)(\bar a \cdots a)\equiv \bar a \cdots a$ and
$S(\langle d_1 \bar{d_0}\rangle^{n})=(m, (n-1)\langle 2m\rangle, m-1)$.
Set $w_y$ to be the reduced alternating word $\langle d_1 \bar{d_0}\rangle^{n}$.
Then $\yr=w_y \bar b \bar w_y$ and
$w_y \bar b \bar w_y\equiv (\bar a \cdots a)\bar b (\bar a \cdots a)
\equiv \bar a \cdots a$ is an alternating word.
Hence $\yr\equiv w_y \bar b \bar w_y$.
Moreover, we can observe that $S(w_y)$ and $S(\yr)$ are of the desired form.
These arguments also work even when $n=1$, if we discard
the entry $(n-1)\langle 2m\rangle=0\langle 2m\rangle$.

Suppose next that $m$ is odd.
Then $\yr= \langle \bar d_1 d_0\rangle^{n} \bar b \langle \bar d_1 d_0\rangle^{-n}$
in $F(a,b)$ by Lemma~\ref{lem:pre_longest-base}.
Since $\bar d_1=\langle \bar a\bar b\rangle^m \langle ab\rangle^{m-1}$
and $d_0=\langle ba\rangle^{m+1}\langle \bar b\bar a\rangle^m$
by Lemma~\ref{lem:longest-base}, we see
\[
 \langle \bar d_1 d_0\rangle^{n}
 \equiv
 \begin{cases}
 \bar a \cdots \bar a\bar b
 &
 \text{if $m\ge 3$ and $n$ is even;}\\
 \bar a \cdots a
 &
 \text{if $m\ge 3$ and $n$ is odd;}\\
  \bar a \cdots a\bar b
 &
 \text{if $m=1$ and $n$ is even;}\\
  \bar a \cdots \bar a
 &
 \text{if $m=1$ and $n\ge 3$ is odd;}\\
\bar a
 &
 \text{if $m=1$ and $n=1$.}
\end{cases}
\]
Set $w_y$ to be the alternating word
$\langle \bar d_1 d_0\rangle^{n}$ or
the alternating word obtained from $\langle \bar d_1 d_0\rangle^{n}$
by deleting the last letter $\bar b$,
according to whether $\langle \bar d_1 d_0\rangle^{n}$ ends with
$a^{\pm 1}$ or $\bar b$.
Then we can see as in the previous cases that
the desired results hold.
\end{proof}

\section{Proof of Theorem~\ref{thm:main_theorem}}
\label{sec:main_result}

In this section, we prove Theorem~\ref{thm:main_theorem}
by using the small cancellation theory.
For standard terminologies in the small cancellation theory,
we refer the readers to \cite[Chapter~V]{lyndon_schupp} and
\cite[Sections~5 and 6]{lee_sakuma}.

Recall the presentation $G(K(r))\cong \langle a,b \svert u_r \rangle$,
where $u_r$ is the cyclically reduced word in $a$ and $b$
given by the formula (\ref{group-presentation}).
Then the symbol $S(r)$ (resp., $CS(r)$) denotes the
$S$-sequence $S(u_r)$ of $u_r$
(resp., cyclic $S$-sequence $CS(u_r)$ of $(u_r)$).
In \cite{lee_sakuma},
we have proved that the sequence $S(r)$ has a canonical decomposition
$(S_1, S_2, S_1, S_2)$
and established various properties of the decomposition.
We summarize the key facts which are used in the proof of Theorem~\ref{thm:main_theorem}.
In the remainder of this section,
by a {\it piece}, we mean
a piece relative to the symmetrized set of relators
generated by $u_r$ (see \cite[Definition~5]{lee_sakuma}).

\begin{proposition}
\label{prop:S-sequence}
The canonical decomposition $(S_1, S_2, S_1, S_2)$
of the sequence $S(r)$ satisfies the following conditions.

\begin{enumerate} [\indent \rm (1)]
\item Each $S_i$ is symmetric and occurs only twice in
the cyclic sequence $CS(r)$.

\item
If $v$ is a subword of the cyclic word $(u_r)$
which is a product of $3$ pieces
but is not a product of $t$ pieces with $t<3$,
then $v$ contains a subword, $v'$,
such that $S(v')=(S_1,S_2, \ell)$
or $S(v')=(\ell, S_2,S_1)$,
for some $\ell\in\ZZ_+$.

\item
Suppose $r=[2m, 2n]$.
Then $CS(r)=\lp S_1,S_2,S_1,S_2\rp$
with $S_1=(2m+1)$ and $S_2=((2n-1)\langle 2m\rangle)$.
Moreover, if $v$ is a subword of $(u_r^{\pm1})$
such that $S(v)=(1,\ell)$ or $(\ell, 1)$
for some $\ell$ with $1\le \ell \le m$,
then $v$ is a piece.

\item
Suppose $r=[2m, -2n]=[2m-1,1,2n-1]$.
Then $CS(r)=\lp S_1,S_2,S_1,S_2\rp$
with $S_1=((2n-1)\langle 2m\rangle)$ and $S_2=(2m-1)$.
Moreover, if $v$ is a subword of $(u_r^{\pm1})$
such that $S(v)$ is of one of the following form,
then $v$ is a piece:
$\ell$ with $1\le \ell \le 2m$,
$(1,\ell)$ with $1\le \ell \le m$,
$(\ell, 1)$ with $1\le \ell \le m$,
$(k \langle 2m \rangle,1)$ with $0\le k\le 2n-2$,
$(1, k \langle 2m \rangle,1)$ with $0\le k\le 2n-3$,
or $(1,k \langle 2m \rangle)$ with $0\le k\le 2n-2$
\end{enumerate}
\end{proposition}

\begin{proof}
(1) This is a part of \cite[Proposition~4.5]{lee_sakuma}.

(2) This follows from the proof of
\cite[Lemma~3.3]{lee_sakuma_7}.
In the lemma, pieces of the symmetrized set of relators
generated by a power $u_r^k$ with $k\ge 2$ of the relator $u_r$
is treated.
However, the same argument also works when $k=1$.
(See also \cite[the proof of Corollary~3.25]{lee_sakuma_2}.)

(3), (4)
The first assertions are nothing other than \cite[Lemma~3.16(1),(3)]{lee_sakuma_2}.
The second assertions follow from the characterization of pieces
described in \cite[Lemma~5.3(2-c)]{lee_sakuma}
and \cite[Lemma~5.2]{lee_sakuma}.
\end{proof}

\begin{proof}[Proof of Theorem~\ref{thm:main_theorem}]
Suppose on the contrary that
the subgroup of $G(K(r))$ with $r=[2m,\pm 2n]$
generated by $\xr$ and $\yr$ is not a free group.
Then there is a nontrivial relation consisting of $\xr$ and $\yr$.
We may assume after conjugacy that
$\xr^{k_1}\yr^{l_1} \cdots \xr^{k_t}\yr^{l_t}=1$
in $G(K(r))$, where each $k_i$ and $l_i$ are non-zero integers.
By Propositions~\ref{prop:longest-base}(3) and~\ref{prop:longest-base_PLUS}(3),
the word $\xr^{k_1}\yr^{l_1} \cdots \xr^{k_t}\yr^{l_t}$
is represented by the following cyclically reduced word $w$ in $\{a,b\}$:
\begin{align}
\label{word-w}
w:\equiv w_x a^{k_1}\bar w_x w_y \bar b^{l_1} \bar w_y \cdots
w_x a^{k_t}\bar w_x w_y \bar b^{l_t} \bar w_y.
\end{align}
Since $w=1$ in $G(K(r))$, there is a reduced van Kampen diagram $(M,\phi)$
over $G(K(r))=\langle a, b | \, u_r \rangle$ such that
$(\phi(\partial M))\equiv (w)$.
Namely, $M$ is a map, i.e.,
a finite $2$-dimensional cell complex
embedded in $\RR^2$,
and $\phi$ is a function assigning to each oriented edge $e$ of $M$, as a {\it label},
a reduced word $\phi(e)$ in $\{a,b\}$ such that the following conditions are satisfied.
\begin{enumerate}[\indent \rm (i)]
\item If $e$ is an oriented edge of $M$ and $e^{-1}$ is the oppositely oriented edge,
then $\phi(e^{-1})\equiv \phi(e)^{-1}$.

\item For any boundary cycle $\delta$ of any face of $M$,
$\phi(\delta)$ is a cyclically reduced word
such that $(\phi(\delta))\equiv (u_r^{\pm 1})$.
(If $\alpha=e_1, \dots, e_n$ is a path in $M$, we define $\phi(\alpha) \equiv \phi(e_1) \cdots \phi(e_n)$.)
\end{enumerate}

By \cite[Corollary~6.2]{lee_sakuma}, $M$ is a $[4,4]$-map
(cf. \cite[Definition~7]{lee_sakuma}).
Then by the Curvature Formula of
Lyndon and Schupp (see \cite[Corollary~V.3.4]{lyndon_schupp}),
we have
\[
\sum_{v \in \, \partial M} (3-d_M(v)) \ge 4,
\]
where $d_M(v)$ is the degree of a vertex $v\in \partial M$ in $M$.
This inequality yields the following Claim~1
(cf. \cite[Claim in the proof of Theorem~6.3]{lee_sakuma}.

\medskip
\noindent{\bf Claim 1.}
{\it In $ \partial M$,
there exist at least four more vertices of degree $2$
than vertices of degree at least $4$.}

\medskip
\noindent{\bf Claim 2.}
{\it Any two of degree $2$ vertices cannot lie consecutively on $\partial M$.}

\begin{proof}[Proof of Claim~2]
Suppose the contrary.
Then $(\phi(\partial M))\equiv (w)$ contains a subword $z$ such that
$z$ is a subword of $(u_r^{\pm 1})$
which is a product of three pieces but is not a product of $t$ pieces with $t\le 2$
(see \cite[Convention~1]{lee_sakuma}).
Note that $z$ is a subword of the cyclic word $(u_r^{\pm 1})$
and hence it is alternating.
On the other hand, we see from
Propositions~\ref{prop:longest-base}(3) and \ref{prop:longest-base_PLUS}(3)
that the words $w_x a^{k}\bar w_x w_y b^{-l} \bar w_y$ or
$w_y b^{-l} \bar w_y w_x a^{k}\bar w_x$ with $k,l\ne 0$
are alternating if and only if $|k|=|l|=1$.
Hence $z$ is a subword of the cyclically alternating cyclic word
$(w'):\equiv (\xr^{\varepsilon_{1,x}}\yr^{\varepsilon_{1,y}} \cdots \xr^{\varepsilon_{t,x}}\yr^{\varepsilon_{t,y}})$,
where $\varepsilon_{i,x}=k_i/|k_i|$ and $\varepsilon_{i,y}=l_i/|l_i|$.
We will show that this cannot be possible in each case.

\medskip
\noindent {\bf Case 1:} $r=[2m,2n]$.
By Proposition~\ref{prop:longest-base}(2),
$\xr\equiv \bar{a}\cdots a$ and $\yr\equiv b\cdots \bar{b}$.
Thus, in the cyclic $S$-sequence of $w'$,
the last component, $m$, of $S(\xr^{\varepsilon_{i,x}})$
and the first component, $m$, of $S(\yr^{\varepsilon_{i+1,y}})$
are amalgamated into a component $2m$.
Similarly, the first component, $m$, of $S(\xr^{\varepsilon_{i,x}})$
and the last component, $m$, of $S(\yr^{\varepsilon_{i-1,y}})$
are amalgamated into a component $2m$.
Hence, we see by using Proposition~\ref{prop:longest-base}(2) that
\begin{equation*}
CS(w')=\lp (2n-1)\langle 2m\rangle, (m+1, m)^{\varepsilon_1},
\dots, (2n-1)\langle 2m\rangle, (m+1, m)^{\varepsilon_{2t}} \rp,
\end{equation*}
where each $\varepsilon_i$ is either $1$ or $-1$
and $(m+1, m)^{-1}$ denotes $(m, m+1)$.
Since $z$ is a product of three pieces
and is not a product of two pieces,
we see by Proposition~\ref{prop:S-sequence}(2),(3) that
$S(z)$ contains $(2m+1,(2n-1)\langle 2m\rangle, \ell)$
or $(\ell, (2n-1)\langle 2m\rangle, 2m+1)$ as a subsequence, for some $\ell \in \ZZ_+$.
This implies that $CS(w')$ contains a term bigger than or equal to $2m+1$,
since $z$ is a subword of $w'$.
But this is an obvious contradiction to the above formula for $CS(w')$.

\medskip
\noindent {\bf Case 2.a:}
$r=[2m,-2n]$, where $m \ge 2$.
We see by using Proposition~\ref{prop:longest-base_PLUS}(2)
as in the previous case that
\begin{equation*}
CS(w')=
\lp (2n-1)\langle 2m\rangle, (m, m-1)^{\varepsilon_1},
\dots, (2n-1)\langle 2m\rangle, (m, m-1)^{\varepsilon_{2t}} \rp,
\end{equation*}
where each $\varepsilon_i$ is either $1$ or $-1$.
Since $z$ is a product of three pieces and is not a product of two pieces,
we see by Proposition~\ref{prop:S-sequence}(2),(4) that
$S(z)$ contains $((2n-1)\langle 2m\rangle, 2m-1, \ell)$
or $(\ell, 2m-1,(2n-1)\langle 2m\rangle)$ as a subsequence, for some $\ell \in \ZZ_+$.
This implies that $CS(w')$ contains a term $2m-1$,
since $z$ is a subword of $w'$.
But this is an obvious contradiction to the above formula for $CS(w')$.

\medskip
\noindent {\bf Case 2.b:} $r=[2m,-2n]$, where $m=1$.
By Proposition~\ref{prop:longest-base_PLUS}(2),
$\xr\equiv b\cdots \bar{b}$ and $\yr\equiv \bar{a}\cdots a$.
Thus as in the previous case,
we see by using Proposition~\ref{prop:longest-base_PLUS}(2),
\begin{equation*}
CS(w')=
\begin{cases}
\lp (2n-2)\langle 2\rangle, 3^{\varepsilon_1},
\dots, (2n-2)\langle 2\rangle, 3^{\varepsilon_{2t}} \rp
& \textrm{if $n \ge 2$};\\
\textrm{every term is $2$, $3$, or $4$} & \textrm{if $n=1$}.
\end{cases}
\end{equation*}
Since $z$ is a product of three pieces and is not a product of two pieces,
we see by Proposition~\ref{prop:S-sequence}(2),(4) that
$S(z)$ contains $((2n-1)\langle 2\rangle, 1, \ell)$
or $(\ell, 1,(2n-1)\langle 2\rangle)$ as a subsequence, for some $\ell \in \ZZ_+$.
This implies that $CS(w')$ contains a term $1$,
since $z$ is a subword of $w'$.
But this is an obvious contradiction to the above formula for $CS(w')$.
\end{proof}

By Claims~1 and 2, there must be some pair of degree $2$ vertices on $\partial M$
having only degree $3$ vertices between them.
Decompose $\partial M$ into paths:
\begin{equation}
\label{equ:decomposition}
\partial M=p_1 q_1 \cdots p_{s} q_{s},
\end{equation}
where every vertex lying in the closure of each $q_i$ has degree $3$
and every vertex lying in the interior of each $p_i$ has degree $2$ or
degree at least $4$.
Here some $q_i$ may be degenerate to a vertex.

Note that $\phi(p_1 q_1 \cdots p_{s} q_{s})$ is not alternating
at $\phi(q_i)$ in the sense that
(i) the last letter of $\phi(p_i)$ and the first letter of $\phi(p_{i+1})$
are the same letter, $a^{\pm 1}$ or $b^{\pm 1}$,
and (ii) $\phi(q_i)$ is equal to $a^{\pm k}$ or $b^{\pm k}$ with $k\ge 0$
accordingly.
On the other hand, $w$ is not alternating precisely at the subwords
$a^{k_j}$ with $|k_j|\ge 2$ and $\bar b^{l_j}$ with $|l_j|\ge 2$
in the expression (\ref{word-w}).
Hence $\phi(q_j)$ corresponds to (possibly empty) subword of
$a^{k_j}$ with $|k_j|\ge 2$ or $\bar b^{l_j}$ with $|l_j|\ge 2$.

Recall that there is a pair of degree $2$ vertices on $\partial M$
having only degree $3$ vertices between them.
After a cyclic permutation of indices,
we may assume that this occurs at $p_1q_1p_2$,
namely the last (resp., first) occurring vertex in the interior of $p_1$ (resp., $p_2$)
has degree $2$.
Then a terminal (resp., initial) subword of $\phi(p_1)$ (resp., $\phi(p_2)$)
is a subword of the cyclic word $(u_r^{\pm1})$
which is a product of two pieces but is not a piece in itself
(see \cite[Convention~1]{lee_sakuma}).
Thus we may assume that the following holds.
(The other possibility that $\phi(p_1q_1p_2)$ contains the subword
$\bar w_x w_y \bar b^{l_j} \bar w_y w_x$ with $|l_j|\ge 2$
can be settled by using Lemma~\ref{lem:group-symmetry}.)
\begin{enumerate}[(i)]
\item
$\phi(p_1q_1p_2)$ contains the subword
$\bar w_y w_x a^{k_j} \bar w_x w_y$ of $w$,
such that $|k_j|\ge 2$.
\item
$\phi(p_1)$ ends with the alternating subword $\bar w_y w_x a^{\varepsilon_{j,x}}$.
\item
$\phi(q_1)\equiv a^{\varepsilon_{j,x}|k_j-2|}$.
\item
$\phi(p_2)$ begins with the alternating subword $a^{\varepsilon_{j,x}} \bar w_x w_y$.
\end{enumerate}
From this, we will derive a contradiction in each case.

\medskip
\noindent {\bf Case 1:} $r=[2m,2n]$.
By using Proposition~\ref{prop:longest-base}(1), we can see
\begin{align*}
S(\bar w_y w_x a^{\varepsilon_{j,x}})
=
\begin{cases}
(m, (2n-1)\langle 2m\rangle, m, 1)
&
\text{if $(-1)^n= \varepsilon_{j,x}$;}\\
(m, (2n-1)\langle 2m\rangle, m+1)
&
\text{if $(-1)^n\ne \varepsilon_{j,x}$.}
\end{cases}
\end{align*}
Hence $S(\phi(p_1))$ ends with
$(m+\ell, (2n-1)\langle 2m\rangle, m, 1)$ or
$(m+\ell, (2n-1)\langle 2m\rangle, m+1)$
for some $\ell\ge 0$
according to whether $(-1)^n= \varepsilon_{j,x}$ or not.
Similarly, by using Proposition~\ref{prop:longest-base}, we can see
\begin{align*}
S(a^{\varepsilon_{j,x}} \bar w_x w_y)
=
\begin{cases}
(m+1, (2n-1)\langle 2m\rangle, m)
&
\text{if $(-1)^n= \varepsilon_{j,x}$;}\\
(1, m,(2n-1)\langle 2m\rangle, m)
&
\text{if $(-1)^n\ne \varepsilon_{j,x}$.}
\end{cases}
\end{align*}
Hence $\phi(p_2)$ begins with
$(m+1, (2n-1)\langle 2m\rangle, m+\ell)$ or
$(1, m,(2n-1)\langle 2m\rangle, m+\ell)$ for some $\ell\ge 0$
according to whether $(-1)^n= \varepsilon_{j,x}$ or not.

Thus we have shown that the following hold.
\begin{enumerate}[(i)]
\item
If $(-1)^n= \varepsilon_{j,x}$, then
$S(\phi(p_1))$ ends with
$((2n-1)\langle 2m\rangle, m, 1)$.
\item
If $(-1)^n\ne \varepsilon_{j,x}$, then
$\phi(p_2)$ begins with
$(1, m,(2n-1)\langle 2m\rangle)$.
\end{enumerate}
This leads to a contradiction as follows.
Suppose that $(-1)^n= \varepsilon_{j,x}$ and so $S(\phi(p_1))$ ends with
$((2n-1)\langle 2m\rangle, m, 1)$.
Recall that $\phi(p_1)$ ends with a subword, say $v$,
of $(u_r^{\pm1})$ which is a product of two pieces
but is not a piece.
Since $CS(r)=(2m+1, (2n-1)\langle 2m\rangle, 2m+1, (2n-1)\langle 2m\rangle)$,
we see $S(v)=(k,1)$ for some $k$ with $1\le k\le m$.
But then $v$ is a piece
by Proposition~\ref{prop:S-sequence}(3), a contradiction.
We also have a similar contradiction
in the remaining case when
$(-1)^n\ne \varepsilon_{j,x}$ and so
$\phi(p_2)$ begins with
$(1, m,(2n-1)\langle 2m\rangle)$.

\medskip
\noindent {\bf Case 2.a:}
$r=[2m,-2n]$, where $m \ge 2$.
As in Case~1,
by using Proposition~\ref{prop:longest-base_PLUS}(1)
and the facts that
$\phi(p_1)$ ends with $\bar w_y w_x a^{\varepsilon_{j,x}}$
and
$\phi(p_2)$ begins with $a^{\varepsilon_{j,x}} \bar w_x w_y$,
we can see
that either
$S(\phi(p_1))$ ends with $((2n-1)\langle 2m\rangle, m-1, 1)$
or
$S(\phi(p_2))$ begins with $(1, m-1, (2n-1)\langle 2m\rangle)$.
Noting that $CS(r)=\lp (2n-1) \langle 2m \rangle, 2m-1, (2n-1) \langle 2m \rangle, 2m-1\rp$,
this implies that
there is a subword, $v$, of $(u_r^{\pm 1})$
which is a product of two pieces but is not a piece,
such that $S(v)=(k,1)$ or $(1,k)$ with $k \le m-1$.
But then $v$ is a piece by Proposition~\ref{prop:S-sequence}(4),
a contradiction.

\medskip
\noindent {\bf Case 2.b:} $r=[2m,-2n]$, where $m=1$ and $n\ge 2$.
As in the previous case, we can see that
either
$S(\phi(p_1))$ ends with $((2n-2)\langle 2\rangle, 1)$
or
$S(\phi(p_2))$ begins with $(1, (2n-2)\langle 2\rangle)$.
Noting that $CS(r)=\lp (2n-1) \langle 2 \rangle,1, (2n-1) \langle 2 \rangle, 1\rp$,
this implies that
there is a subword, $v$, of $(u_r^{\pm 1})$
which is a product of two pieces but is not a piece,
such that $S(v)=(k \langle 2 \rangle,1)$ with $0\le k\le 2n-2$,
$(1, k \langle 2 \rangle,1)$ with $0\le k\le 2n-3$,
or $(1,k \langle 2 \rangle)$ with $0\le k\le 2n-2$.
But, this implies that $v$ is a piece by Proposition~\ref{prop:S-sequence}(4),
a contradiction.

\medskip
\noindent {\bf Case 2.c:} $r=[2m,-2n]$, where $m=n=1$.
In this case, by Proposition~\ref{prop:longest-base_PLUS}(1),
$w_x \equiv b$ and $w_y \equiv \bar{a}$.
So $\phi(p_1)$ ends with $a b a^{\varepsilon_{j,x}}$ and
$\phi(p_2)$ begins with $a^{\varepsilon_{j,x}} \bar{b} \bar{a}$.
Suppose that $\varepsilon_{j,x}=1$. (The other case when $\varepsilon_{j,x}=-1$
can be treated similarly.)
Since $\phi(p_1)$ ends with
a subword, $v_1$, of $(u_r^{\pm 1})$ which is a product of two pieces
but is not a piece, and since $\phi(p_2)$ begins with
a subword, $v_2$, of $(u_r^{\pm 1})$ which is a product of two pieces
but is not a piece, we see by using Proposition~\ref{prop:S-sequence}(4) that
$v_1 \equiv ba$ and $v_2 \equiv a\bar{b}\bar{a}$.
Considering the equality $v_1 \equiv ba$ together with the facts that
$CS(r)=\lp 2,1,2,1\rp$ and
every vertex lying in the closure of $q_1$ has degree $3$,
we see that
three incoming edges of each vertex lying in the closure of $q_1$
must have label $a$, $\bar a$ and $b$, respectively.
But then $bv_2 \equiv ba\bar{b}\bar{a}$ is a subword of $(u_r^{\pm 1})$,
which is a contradiction to $CS(r)=\lp 2,1,2,1\rp$.

The proof of Theorem~\ref{thm:main_theorem} is now completed.
\end{proof}

\section{Proof of Theorem~\ref{thm:main_theorem_2}}
\label{sec:main_result_2}

In this section, we prove Theorem~\ref{thm:main_theorem_2}
by using the homology of the double branched covering, $M(K(r))$,
of $S^3$ branched over $K(r)$
and the $\pi$-orbifold group $\OO(K(r))$
introduced by Boileau and Zimmermann~\cite{Boileau-Zimmermann}.

As in \cite[Section~2]{lee_sakuma},
we regard $(S^3,K(r))$ as the union of two rational tangles
$\rtangle{\infty}$ and $\rtangle{r}$
of slopes $\infty$ and $r$.
Here the common boundary
$\partial \rtangle{\infty}=\partial \rtangle{r}$
is identified with the Conway sphere $\Conways:=(\RR^2,\ZZ^2)/H$,
where $H$ is the group of isometries
of the Euclidean plane $\RR^2$
generated by the $\pi$-rotations around
the points in the lattice $\ZZ^2$.
For each rational number $s\in\QQQ=\QQ\cup\{\infty\}$,
a line of slope $s$ in $\RR^2-\ZZ^2$
projects to an essential simple loop,
denoted by $\alpha_s$, in $\PConway:=\Conway-\PP$.
Similarly,
a line of slope $s$ in $\RR^2$ passing through a point $\ZZ^2$
determines an essential simple proper arc,
denoted by $\delta_s$, in $\PConway:=\Conway-\PP$.
The rational number $s$ is called the {\it slope} of $\alpha_s$ and $\delta_s$.
By the definition of the rational tangles,
the loops $\alpha_{\infty}$ and $\alpha_r$ bound disks
in $B^3-t(\infty)$ and $B^3-t(r)$, respectively.

The double branched covering $M(K(r))$ of $(S^3,K(r))$
is the union of the solid tori $V_{\infty}$ and $V_r$
which are obtained as the double branched coverings
of $\rtangle{\infty}$ and $\rtangle{r}$, respectively.
Let $\tilde\alpha_0$ and $\tilde\alpha_{\infty}$ be lifts in $\partial V_{\infty}$ of
the simple loops $\alpha_0$ and $\alpha_{\infty}$, respectively.
Then $\tilde\alpha_0$ and $\tilde\alpha_{\infty}$
form the meridian and the longitude of $V_{\infty}$.
Similarly a lift $\tilde\alpha_r$ of $\alpha_r$ in $\partial V_r$ is a meridian of $V_r$.
Thus
$[\tilde\alpha_{\infty}]$ and $[\tilde\alpha_{r}]$ are the zero elements of
 $H_1(V_{\infty})$ and $H_1(V_r)$, respectively.
Since $[\tilde\alpha_{r}]=p[\tilde\alpha_0]+q[\tilde\alpha_{\infty}]$ in
$H_1(\partial V_{\infty})$, where $r=q/p$, we have
\begin{align*}
H_1(M(K(r))
\cong \langle \tilde\alpha_{0} \ | \
p[\tilde\alpha_0]\rangle
\cong
\ZZ/p\ZZ
\end{align*}

Recall the $\pi$-orbifold group $\OO(K(r))$ of the knot $K(r)$,
which is defined as the quotient of the knot group $G(K(r))$
by the normal subgroup normally generated by the square of meridians
(see \cite{Boileau-Zimmermann}).
Then $\OO(K(r))$ is the semidirect product
\[
\pi_1(M(K(r))) \rtimes \ZZ/2\ZZ \cong
H_1(M(K(r))) \rtimes \ZZ/2\ZZ
\cong
\ZZ/p\ZZ \rtimes \ZZ/2\ZZ
\]
and so it is isomorphic to the dihedral group of order $2p$.

\begin{figure}[h]
\includegraphics{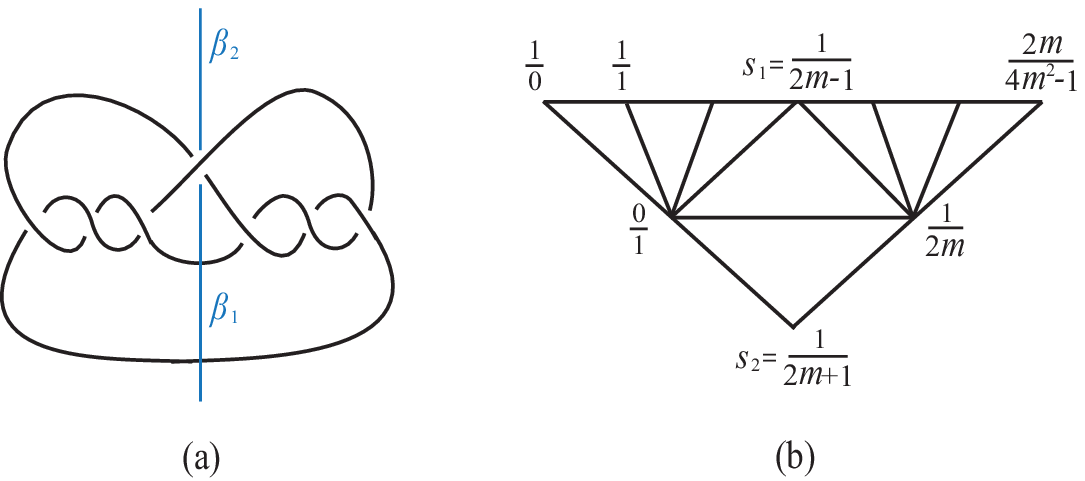}
\caption{\label{fig.additional-inversion2}}
\end{figure}

We prove Theorem~\ref{thm:main_theorem_2}
by showing that the images in $\OO(K(r))$
of the groups generated by the meridian pairs in the theorem
is a proper subgroup of $\OO(K(r))$.
Let $r$ be the rational number
\[
[2m,-2m]=[2m-1,1,2m-1]=\frac{2m}{4m^2-1}=\frac{2m}{(2m+1)(2m-1)}.
\]
Observe that the involution $h_3$ in Figure~\ref{fig.additional-inversion}
is equivalent to the involution in Figure~\ref{fig.additional-inversion2}(a).
Let $\beta_1$ and $\beta_2$ be the components
of $\beta-K(r)$ where $\beta=\Fix(h_3)$,
as illustrated in Figure~\ref{fig.additional-inversion2}(a).
We can observe that the arcs $\beta_1$ and $\beta_2$
are the proper essential arcs in the Conway sphere $\PConway$
of slopes $s_1:=1/(2m-1)$ and
$s_2:=1/(2m+1)$,
respectively.
In fact, the involution $h_3$ preserves $\PConway$ and
the involution of the Farey tessellation induced by the restriction of
$h_3|_{\PConway}$ is the reflection
in the geodesic joining $s_1$ and $s_2$
(see Figure~\ref{fig.additional-inversion2}(b)).

For $i=1,2$, let $\{x_i,y_i\}$ be the meridian pair represented by the proper arc $\beta_i$.
Then, by the above observation,
the subgroup $\langle x_i, y_i\rangle$ of $G(K(r))$
is equal to $\langle x_i, \alpha_{s_i}\rangle$,
where $\alpha_{s_i}$ is an element of $G(K(r))$
represented by the simple loop $\alpha_{s_i}$ in $\PConway$ of slope $s_i$.
Note that
\[
[\tilde\alpha_{s_1}]=(2m-1)[\tilde\alpha_0]+1[\tilde\alpha_{\infty}]
=(2m-1)[\tilde\alpha_0]
\in H_1(M(K(r))
\cong \ZZ/(4m^2-1)\ZZ.
\]
Thus the subgroup $\langle [\tilde\alpha_{s_1}]\rangle$
of $H_1(M(K(r)))$
has order $2m+1$ and so it is a proper subgroup of $H_1(M(K(r)))$.
Similarly,
the subgroup $\langle [\tilde\alpha_{s_2}]\rangle$
of $H_1(M(K(r)))$
has order $2m-1$ and so it is a proper subgroup of $H_1(M(K(r)))$.

On the other hand,
the image of $\langle x_i, y_i\rangle=\langle x_i, \alpha_{s_i}\rangle$
in $\OO(K(r))$ is the semidirect product
$\langle [\tilde\alpha_{s_i}]\rangle  \rtimes \ZZ/2\ZZ$.
Hence, it is a proper subgroup of $\OO(K(r))$,
and therefore
$\langle x_i, y_i\rangle$ is a proper subgroup of $G(K(r))$
for each $i=1,2$.
This completes the proof of Theorem~\ref{thm:main_theorem_2}.

\begin{remark}
{\rm
(1) We can show that $\langle x_2, y_2\rangle$ is a free group
by an argument parallel to the proof of Theorem~\ref{thm:main_theorem}.
However, our method does not work for the subgroup
$\langle x_1, y_1\rangle$.

(2)
Theorem~\ref{thm:main_theorem_2}
can be easily extended to
every $2$-bridge knot $K(r)$ with $r=q/p$ such that $q^2\equiv 1 \pmod{p}$.
In fact, we can see that, for the additional strong inversion $h_3$,
the components, $\beta_1$ and $\beta_2$, of $\Fix(h_3)\cap E(K(r))$ are proper essential arcs
of slopes $s_1=q_1/p_1$ and $s_2=q_2/p_2$,
where $p=p_1p_2$ and both $p_1$ and $p_2$ are greater than $1$.
Thus we can see that the subgroup of $\OO(K(r))$
generated by the meridian pair represented by $\beta_i$
is a proper subgroup for $i=1,2$.
}
\end{remark}

\section{Proof of Theorems~\ref{cor:Agol-special} and~\ref{thm:cor-epimorohism}}
\label{sec:main_result_3}

\begin{proof}[Proof of Theorem~\ref{cor:Agol-special}]
Immediate from
Theorems~\ref{thm:main_theorem} and~\ref{thm:main_theorem_2} and
Corollary~\ref{prop:strong-involution_2}.
\end{proof}

\begin{proof}[Proof of Theorem~\ref{thm:cor-epimorohism}]
The if part follows from the if part of \cite[Main Theorem~2.4]{lee_sakuma}
(which is essentially equivalent to \cite[Theorem~1.1]{Ohtsuki-Riley-Sakuma})
and the fact that $G(K(r))$ is isomorphic to $G(K(r'))$.
So we prove the only if part.
Let $\varphi:G(K(\tilde r))\to G(K(r))$ be an epimorphism
between $2$-bridge knot groups satisfying the assumption of the theorem.
By \cite[Corollary~1.3]{BBRW},
$\varphi$ maps
the upper meridian pair $\{\tilde a,\tilde b\}$ of $G(K(\tilde r))$
to peripheral elements of $G(K(r))$.
Thus $\{\varphi(\tilde a), \varphi(\tilde b)\}$ is a parabolic generating pair
and hence by Corollary~\ref{prop:strong-involution_2},
it is either (i) the upper or lower meridian pair,
(ii) the long upper or long lower meridian pair,
or (iii) isomorphic to the upper or lower exceptional pair.
However, Theorems~\ref{thm:main_theorem} and~\ref{thm:main_theorem_2} prohibit
the last two possibilities, and
hence $\varphi$ maps the upper meridian pair of $G(K(\tilde r))$
to the upper or lower meridian pair of $G(K(r))$.

Suppose first that $\varphi$ maps the upper meridian pair of $G(K(\tilde r))$
to the upper meridian pair of $G(K(r))$.
Then
$\tilde r$ or $\tilde r+1$ belongs to the $\hat\Gamma_r$-orbit of $r$ or $\infty$
 by \cite[Main Theorem~2.4]{lee_sakuma}.
Suppose next that $\varphi$ maps the upper meridian pair of $G(K(\tilde r))$
to the lower meridian pair of $G(K(r))$.
Note that
there is an isomorphism from $G(K(r))$ to $G(K(r'))$
which maps the lower meridian pair of $G(K(r))$ to
the upper meridian pair of $G(K(r'))$.
Thus the composition of $\varphi$ and this isomorphism
is an epimorphism from $G(K(\tilde r))$ to $G(K(r'))$
which maps the upper meridian pair of $G(K(\tilde r))$
to that of $G(K(r'))$.
Hence, by \cite[Main Theorem~2.4]{lee_sakuma},
$\tilde r$ or $\tilde r+1$ belongs to the $\hat\Gamma_{r'}$-orbit of $r'$ or $\infty$.
This completes the proof of Theorem~\ref{thm:cor-epimorohism}.
\end{proof}

\bibstyle{plain}

\end{document}